\documentclass{amsart}
\usepackage[utf8]{inputenc}
\usepackage[foot]{amsaddr}

\usepackage{scrextend}
\usepackage[a4paper, margin=1in, bottom = 1.3in]{geometry}
\usepackage{color}
\usepackage[utf8]{inputenc}
\usepackage{amsmath}
\usepackage{amsfonts}
\usepackage{asymptote}
\usepackage{pifont}
\usepackage{float}

\usepackage{fancyhdr} 
\usepackage{blindtext}

\usepackage{amsmath,amsthm,amssymb}
\usepackage{amscd}

\newtheorem{Thm}{Theorem}[section]
\newtheorem{Prop}[Thm]{Proposition}
\newtheorem{Lem}[Thm]{Lemma}

\theoremstyle{definition}
\newtheorem{Rem}[Thm]{Remark}
\newtheorem{Def}[Thm]{Definition}

\newcommand{\op}{\operatorname}

\DeclareMathOperator{\Der}{Der}
\DeclareMathOperator{\tr}{tr}
\DeclareMathOperator{\ad}{ad}

\DeclareMathOperator{\dv}{div}

\title{The Elliptic Kashiwara-Vergne Lie algebra in low weights}
\author[F. Naef]{Florian~Naef}
\address{Florian~Naef: Department of Mathematics, Massachusetts Institute of Technology, 182 Memorial Dr, Cambridge, MA 02142, USA}
\email[Florian~Naef]{naeffl@mit.edu}
\author[Y. Qin]{Yuting~Qin}
\address{Yuting~Qin: The Webb Schools, Claremont, California, 91711, USA.}
\email[Yuting~Qin]{emmaqin@mit.edu}

\begin{document}
\maketitle
\begin{abstract}
In this paper, we study the elliptic Kashiwara-Vergne Lie algebra $\mathfrak{krv}$, which is a certain Lie subalgebra of the Lie algebra of derivations of the free Lie algebra in two generators. It has a natural bigrading, such that the Lie bracket is of bidegree $(-1,-1)$. After recalling the graphical interpretation of this Lie algebra, we examine low degree elements of $\mathfrak{krv}$. More precisely, we find that $\mathfrak{krv}^{(2,j)}$ is one-dimensional for even $j$ and zero for $j$ odd. We also compute $\operatorname{dim}(\mathfrak{krv})^{(3,j)} = \lfloor\frac{j-1}{2}\rfloor - \lfloor\frac{j-1}{3}\rfloor$. In particular, we show that in those degrees there are no odd elements and also confirm Enriquez' conjecture in those degrees.
\end{abstract}


\section{Introduction}
The elliptic Kashiwara-Vergne Lie Algebra $\mathfrak{krv} = \mathfrak{krv}_{ell}$ was originally defined in \cite{AKKN} as the symmetry Lie algebra of the Goldman-Turaev Lie bialgebra of a genus 1 surface with 1 boundary component. More precisely, on the vector space spanned by homotopy classes of free loops on the once-punctured torus, Goldman and Turaev define a natural Lie bialgebra structure in terms of intersections and self-intersections of loops (see \cite{AKKN} and references therin). It can be shown to be isomorphic (after a certain completion) to one of Schedler's necklace Lie bialgebras associated to the quiver with one vertex and one loop. Then $\mathfrak{krv}$ can be defined as automorphisms of the completed path algebra on the quiver respecting the Lie bialgebra structure. In \cite{AKKN} it is computed to be
$$
\mathfrak{krv} = \{u \in \Der(L(x, y)) \ |\ u([x, y]) = 0, \ \dv(u) = f([x,y]) \text{ for some $f$}\},
$$
where the divergence operator $\dv$ is explained below. This is closely related to Enriquez's elliptic Grothendieck-Teickmüller Lie algebra $\mathfrak{grt}^\text{ell}$ as defined in \cite{Enriquez}, that appears in the study of elliptic Drinfeld associators. It is conjectured that there is an isomorphism $\mathfrak{grt}^\text{ell} \to \mathfrak{krv}$. Moreover, Enriquez conjectures that $\mathfrak{grt}^\text{ell}$ is generated by the two infinite families given by the Drinfeld generators $\sigma_3, \sigma_5, \dots,$ and $\delta_{2}, \delta_4, \dots,$ (definition is recalled below) and a copy of $\mathfrak{sl}_2$. We show that this conjecture is true in weights $2$ and $3$ (or $1$ and $2$ in the convention used in the literature). More concretely, let $\mathfrak{krv}^{i,j}$ denote the set of derivations in $\mathfrak{krv}$ that increase the $x$ and $y$ degree by $i-1$ and $j-1$, respectively. We find the following

\begin{Thm}
\label{thm:main}
\begin{align*}
\mathfrak{krv}^{2,j} &= \begin{cases}
	\mathbb{R} \delta_{j}, &\text{for } j \text{ even} \\
0, &\text{for } j \text{ odd},
\end{cases}\\
\op{dim}\mathfrak{krv}^{3,j} &= \begin{cases}
	0, &\text{for } j \text{ even} \\
\lfloor\frac{j-1}{2}\rfloor - \lfloor\frac{j-1}{3}\rfloor, &\text{for } j \text{ odd}.
\end{cases}
\end{align*}
\end{Thm}

In Section 2, we recall the basic definitions concerning the space $F(A)$ of symplectic derivation of the free associative algebra $A = \mathbb{R}\langle x,y\rangle$. In Section 3, we recall how one obtains the symplectic derivation Lie subalgebra $F(L) \subset F(A)$ of the free Lie algebra $L \subset A$. We also recall the definitions and properties of partial differential operators on these spaces. In Section 4, we recall the definition of the elliptic $\mathfrak{krv}$ and give some of its properties. In Section 5, we study low degree elements in $\mathfrak{krv}_{ell}$ and show Theorem \ref{thm:main}.

\section*{Acknowledgements}
We would like to thank PRIMES MIT during which this work was completed. In particular, we thank Tanya Khovanova and Pavel Etingof for their help and encouragement. We also thank Anton Alekseev and Leila Schneps for useful discussions.

\section{Symplectic Derivations of $A$}
Let thoughout the paper
$$
A := \mathbb{R}\langle x,y\rangle,
$$
denote the free associative algebra in two generators $x$ and $y$. That is, elements in $A$ are linear combinations of monomials $\prod{x_i},$ where $x_i \in \{x, y\}$. Following \cite{Kontsevich} we introduce the vector space
$$
F(A) := A \otimes A / \{a \otimes b - b \otimes a,  a \otimes bc - ab \otimes c, \ \forall a,b,c \in A \}.
$$
That is the vector space quotient of $A \otimes A$ by the relations $a \otimes b - b \otimes a, \ \forall a,b \in A$ and $a\otimes bc - ab \otimes c, \ \forall a,b,c \in A$. We also define the \emph{trace space},
$$
\tr := A / \{ ab - ba \ | \ a,b \in A\},
$$
and denote the canonical projection $A \to \tr$ by $a \mapsto \tr(a)$.

The following lemma shows that we can identify the spaces $F(A)$ and $\tr$. In particular, there is a canonical map $\tr: A \to F(A)$.
\begin{Lem}
The assignment $f: \ F(A) \ni a\otimes b \to \tr(ab) \in \tr$ defines an isomorphism of vector spaces. 
\end{Lem}

\begin{proof}
We first show that $f$ is well defined. That is, the assignment $a \otimes b \mapsto \tr(ab)$ sends the relations $a \otimes b - b \otimes a$ and $a \otimes bc - ab \otimes c$ to $\tr(ab) - \tr(ba) = 0$ and $\tr(abc) - \tr(abc) =$, respectively.
An inverse to $f$ is given by $\tr(a) \mapsto 1 \otimes a \in F(A)$. Is well-defined since $1 \otimes ab - 1 \otimes ba = (1 \otimes ab - 1 a \otimes b) + (a \otimes b  - b \otimes a)+ (1b \otimes a- 1 \otimes ba)$. It is an inverse since $\tr(a) \mapsto 1 \otimes a \mapsto \tr(a)$ and $a \otimes b \mapsto \tr(ab) \mapsto 1 \otimes ab = a \otimes b + (1 \otimes ab - 1 a \otimes b)$.
\end{proof}

\begin{Rem}
We can visualize monomials in $A$ as markings on a line, where each marking is labeled by either $x$ or $y$. To multiply two elements in $A,$ we place them next two each other and combine the line. Elements of the trace space $\tr$, or cyclic words, we think of as cirlces with markings labelled by $x$ and $y$. The trace map $A \to \tr$ is then given by joining the ends of the line segment into a circle. The relation $\tr(ab) - \tr(ba) = 0$ says that we can rotate the circle. The presentation of $\tr$ as $F(A)$ is obtained by gluing together a circle from two line segments. The relations correspond to the fact that it does not matter which line segment is first and which one is the second, and that we can split off part of one line segment and join it to the other.
\end{Rem}

We recall the definition of a derivation of an associative algebra. Let $M$ be an $A$-bimodule.
\begin{Def}
A linear map $u:\; A \longrightarrow M$ is a \emph{derivation} if $\forall a, b \in A$
$$
u(ab) = u(a)\;b + a\;u(b), \quad \forall a,b \in A.
$$
We denote the set of derivations of $A$ into $M$ by $\Der(A,M)$ and define $\Der(A) = \Der(A,A)$.
\end{Def}

Since our algebra is free on generators $x$ and $y$ ther are two canonical derivations $A \to A \otimes A$.
\begin{Def}
For $x_0 \in \{ x,y \}$ we define the \emph{partial} of an associative term with respect to $x_0$ to be the linear map,
\begin{align*}
    \textbf{$\partial_{x_0}^A$:}  \;\;\;\;\;\;\;\;\;\;\;
    A \; & \longrightarrow A \otimes A, \ \ \text{defined on monomials by}\\
    \prod a_i \; &\longmapsto \sum\limits_{i \;|\; a_i = {x_0}}^{} \; \Big( \prod\limits_{k=1}^{i-1}a_k \;\otimes  \prod\limits_{k=i+1}^{n} a_k \Big)\\
    \text{and linearly }&\text{extended to arbitrary elements in }A
\end{align*}

We denote $\partial_{x_0}^A (a) = \partial_{x_0}^1(a) \otimes \partial_{x_0}^2(a).$
We sometimes write $\partial_{x_0}$ instead of $\partial_{x_0}^A$ when it is clear from the context which partial we are referring to.
\end{Def}

The following shows that our definition of $\partial_{x_0}$ gives indeed derivations $A \to A \otimes A$.
\begin{Prop}
For $f, g \in A,$ we have $\partial_{x_0}(fg) = \partial_{x_0}(f) g + f \partial_{x_0}(g).$
\end{Prop}
\begin{proof}
\begin{align*}
&\partial_{x_0}(fg) = \partial_{x_0}(\prod(f_i) \prod(g_i))\\
=& \sum\limits_{i \ | (fg)_i = {x_0}}^{} \ \Big(\prod\limits_{k=1}^{i-1} ((fg)_k) \;\otimes  \prod\limits_{k=i+1}^{n} ((fg)_k)\Big)\\
=& \sum\limits_{i \;|\; f_i = {x_0}}^{} \; \Big(\prod\limits_{k=1}^{i-1}f_k \;\otimes  \prod\limits_{k=i+1}^{n} f_k \prod(g_i) \Big) + \sum\limits_{i \;|\; g_i = {x_0}}^{} \; \Big( \prod(f_i) \prod\limits_{k=1}^{i-1}g_k \;\otimes  \prod\limits_{k=i+1}^{n} g_k \Big)\\
=&\partial_{x_0}(f) g + f \partial_{x_0}(g).
\end{align*}

\end{proof}

\begin{Rem}
Graphically, the $\partial_{x_0}^A$ operator marks each $x_0$ and replaces it with a tensor product. That is we identify the space of ``marked" monomials in $A$ with $A\otimes A$. The same idea of marking holds when we take the partial of traces, lie terms and lie trees, which will be defined later in the paper. We mark each $x_0$ and identify the space of these objects with one marking as a more familiar space.
\end{Rem}

\begin{Def}
The \emph{partial} of an object in $F(A)$ is given by
\begin{align*}
    \textbf{$\partial_{x_0}^{F(A)}$:}  \;\;\;\;\;\;\;\;\;\;\;
    F(A) \; & \longrightarrow A\\
    \tr(\prod a_i) \; &\longmapsto \sum\limits_{i \;|\; a_i = {x_0}}^{} \; \Big(\prod\limits_{k=i+1}^{n} a_k \; \;  \prod\limits_{k=1}^{i-1}a_k \Big)
\end{align*}
\end{Def}

The map is well-defined since two elements that are equivalent in $F(A)$, $\tr(ab)$ and $\tr(ba)$, are both sent to the same element
$$\sum\limits_{i \;|\; a_i = {x_0}}^{} \; \Big(\prod\limits_{k=i+1}^{n_a} a_k \; \;  b  \prod\limits_{k=1}^{i-1}b_k \Big) + \sum\limits_{i \;|\; b_i = {x_0}}^{} \; \Big(\prod\limits_{k=i+1}^{n_b} b_k \; \;  a  \prod\limits_{k=1}^{i-1}b_k \Big).$$

Moreover, it can be expressed in terms of $\partial^A_{x_0}$.
\begin{Lem}\label{lem:somethingmadeoutofpartial}
For any $a \in A$ we have
$$
\partial_{x_0}^{F(A)} ( \tr(a) ) =  \partial_{x_0}^2(a)\partial_{x_0}^1(a).
$$
\end{Lem}
\begin{proof}
This can be checked readily on monomials from the definitions.
\end{proof}

We wish to show that any derivation $A \to M$ can be expressed in terms of the $\partial^A_{x_0}$'s. For that we need the following
\begin{Def}
We define the following map,
\begin{align*}
    \textbf{\ding{71}\;:} \;\;\;\; \Big( A \otimes A \Big) \times M & \longrightarrow M\\
    (\sum A_i \otimes B_i) \times C & \longmapsto \sum A_i CB_i
\end{align*}
\end{Def}

\begin{Prop}
Let $u$ be a derivation of $\mathbb{R}\langle x,y\rangle$. Then 
$$
u(a) = \partial_x^A (a)\text{\ding{71}}u(x) +  \partial_y^A (a)\text{\ding{71}}u(y) \quad \forall a \in A.
$$
\end{Prop}
\begin{proof}
We proof this by induction on the length of $f \in A$. The claim is true for $x$ and $y.$ If the claim is true for $f$ and $g$, then
\begin{align*}
u(fg) &= u(f)g + f\ u(g) \\
&= (\partial_x^A (f) g)\text{\ding{71}}u(x) + (f \partial_x^A (g))\text{\ding{71}}u(x) +  (\partial_y^A (f) g) \text{\ding{71}}u(y) + (f \partial_y^A (g))\text{\ding{71}}u(y) \\
&= \partial_x^A (fg)\text{\ding{71}}u(x) + \partial_y^A (fg)\text{\ding{71}}u(y).
\end{align*}
\end{proof}

A reformulation of the above formula is given by the following
\begin{Prop}
\label{prop:derivuniqueness}
A derivation is uniquely determined by its values on generators. More precisely, an inverse to the map
\begin{align*}
    \Der(A, M) &\longrightarrow M \times M,\\
u &\mapsto (u(x), u(y)),
\end{align*}
is given by
$$
(u_x,u_y) \mapsto \partial_x^A \text{\ding{71}}u_x +  \partial_y^A \text{\ding{71}}u_y.
$$
%
\end{Prop}
%
%
%
%

As an application we prove the following lemma
\begin{Lem}
\label{lem:euler}
The following identities hold
\begin{gather*}
\partial_x^A(a) \text{\ding{71}} x = N_x a \quad \partial_y^A(a) \text{\ding{71}} y = N_y a \quad \text{for any $a \in A$ of homogenous bidegree $(N_x, N_y)$},\\
\tr( \partial_x^{F(A)}(f) x) = N_x f \quad \tr( \partial_y^{F(A)}(f) y) = N_y f \quad \text{for any $f \in F(A)$ of homogenous bidegree $(N_x, N_y)$},\\
\partial_x^A(a) \text{\ding{71}} (x \otimes 1 - 1 \otimes x) + \partial_y^A(a) \text{\ding{71}} (y \otimes 1 - 1 \otimes y) = a \otimes 1 - 1 \otimes a \quad \text{for any $a \in A$}.
\end{gather*}
\end{Lem}
\begin{proof}
The second equation is a direct property of the first one. For the first one, let us denote the left hand side by $u$. By the previous proposition it is the unique derivation with the property $u(x) = x$ and $u(y) = 0$. We check by induction that it is $u(a) = N_x a$ for $a$ of homogenous degree $N_x$. Let $b,c\in A$ be of $x$-degree $k$ and $l$ respectively, then
$$
u(bc) = u(b)c + bu(c) = k bc + l bc = (k+l) bc,
$$
which completes the proof.
To show the last equation, we note that the left hand side is a derivation $A \to A \otimes A$ that sends $x \mapsto x \otimes 1 - 1 \otimes x$ and $y \mapsto y \otimes 1 - 1 \otimes y$. To show the equality it is enough to check that the right hand side defines a derivation, that is
\begin{align*}
ab \otimes 1 - 1 \otimes ab = a (b \otimes 1 - 1 \otimes b) + (a \otimes 1 - 1 \otimes a) b.
\end{align*}
\end{proof}

Now we are ready to define the space of symplectic derivations.
\begin{Def}
A derivation $u \in \Der(A)$ is called \emph{symplectic} if it sends $\omega = [x,y]$ to 0. We will denote the set of symplectic derivations as $\Der_\omega(A)$.
\end{Def}

Our goal is to show that there is a natural bijection between $F(A)$ and $\Der_\omega(A)$. To that purpose, let us define the linear map
$\Phi_1: F(A) \to \Der(A)$, $f \mapsto u_f$,
where
\begin{align*}
    u_f(x) &=  \partial_y^{F(A)} (f), \\
    u_f(y) &=  -\partial_x^{F(A)} (f).
\end{align*}

\begin{Prop}
$\Phi_1$ defines an isomorphism $F^+(A) = F(A)/\tr(1) \to \Der_\omega(A)$.
\end{Prop}
We split the proof of injectivity and surjectivity into the next two lemmas.

\begin{Lem}
\label{lem:inverseMap}
The kernel of the map $\Phi_1: F(A) \to \Der(A)$ is linearly spanned by $tr(1)$.
\end{Lem}
\begin{proof}
It is clear that $tr(1)$ is in the kernel of $\Phi_1$. To show the converse, we assume that $f \in F(A)$ is such that $u_f = 0$. We can assume that $u_f \in \Der(A)$ is homogeneous of degree $N$, that is $u(x)$ and $u(y)$ have degree $N-1$ (containing $N-1$ symbols). We define
$$
g =  \frac{1}{N}\tr(y u_f(x) -x u_f(y))
$$
and show that $f = g$. Namely,
$$
g = \frac{1}{N} ( \tr(y \partial^{F(A)}_y (f) + x \partial^{F(A)}_x (f)) = f,
$$
where we used Lemma \ref{lem:euler} in the last equality.

\begin{Lem}
The image of $F(A)$ under $\Phi$ is $\Der_{\omega}(A)$, the set of derivations that send $[x,y]$ to 0.
\end{Lem}

First, we show that the image of $F(A)$ is a subset of $\Der_{\omega}(A)$. Let $f \in F(A)$ and $u = u_f$, then we have
\begin{align*}
u([x,y]) &= u(x)\; y - y\; u(x) + x\;u(y) - u(y)\;x \\
&=  \partial_y f\; y - y\; \partial_y f - x\; \partial_x f + \partial_x f \;x.
\end{align*}
To show that this vanishes we apply the map $a\otimes b \to ba$ to the third equation in Lemma \ref{lem:euler} to obtain
\begin{align*}
\partial^2_x(a)\partial^1_x(a)x - x \partial^2_x(a)\partial^1_x(a) + \partial^2_y(a)\partial^1_y(a)y - y \partial^2_y(a)\partial^1_y(a) = 0.
\end{align*}
Using the definition of $\partial^{F(A)}$ in terms of $\partial^A$ (Lemma \ref{lem:somethingmadeoutofpartial}) it follows that for $f = \tr(a)$, we have
$$
u_f([x,y]) = 0.
$$

Next, we show that if a derivation $u$ is in $\Der_{\omega}(A),$ it must be in the image of $F(A)$ under $\Phi.$

A derivation that sends $\omega$ to 0 satisfies $u([x,y]) = 0,$ or $[u(x),y] + [-u(y),x] = 0.$

We define the following maps:
\begin{align*}
    p_1: A &\longrightarrow A\otimes A \otimes A\\
    a &\mapsto \partial_x^A(\partial_x^1(a)) \otimes \partial_x^2(a)\\
    p_2: A &\longrightarrow A\otimes A \otimes A\\
    a &\mapsto \partial_x^A(\partial_y^1(a)) \otimes \partial_y^2(a)\\
    \text{\ding{71}}_{x_i}: A\; \otimes\; &A \otimes A \longrightarrow A\\
    a\;\otimes\; &b \otimes c \; \; \mapsto b {x_i} ca
\end{align*}
Applying \ding{71}$_{x}$ $\circ \; p_1$ to both sides of $[u(x),y] + [-u(y),x] = 0$ gives
\begin{align*}
    \text{\ding{71}}_{x} \ \circ \; p_1 (u(x) \ y - y \ u(x) - u(y) \ x + x \ u(y)) &= 0
\end{align*}
$$-\partial_x^2(u(y)) x \partial_x^1(u(y)) + \partial_x^1(u(y)) x \partial_x^2(u(y)) = 0,$$ or $(n-1) u(y) = \partial_x^2(u(y)) x \partial_x^1(u(y)).$


Applying \ding{71}$_{y}$ $\circ \; p_2$ gives $$\partial_x^2 (u(x)) x \partial_x^1 (u(x)) - \partial_y^1 x \partial_y^2(u(y)) = 0 $$ or $n u(x) = \partial_y^2(u(y)) x \partial_y^1(u(y)).$

Consider the cyclic word $f = \tr(u(y)x).$
Let $u' = \Phi(f).$
$$u'_f(x) = \partial_y^2(u(y)) x \partial_y^1(u(y)) = n u(x)$$
$$u'_f(y) = u(y) + \partial_x^2(u(y)) x \partial_x^1(u(y)) = n u(y).$$

Therefore $u$ must come from the cyclic word $\frac{1}{n} \tr(u(y)x)$ from the map $\Phi.$

\end{proof}

\section{Symplectic Derivations of $L$}
In this section we define the space of symplectic derivations on the free Lie algebra in two generators. To that purpose, we define the free Lie algebra $L$ in two generators  $x$ and $y$ to be the smallest Lie subalgebra of $(A , [a,b] = ab - ba)$ containing $x$ and $y$.
%
%
%
%
The Lie version of $F(A)$ is now given by the following
\begin{Def}
Let $F(L)$ denote the vector space quotient of $L \otimes L$ by the relations $a \otimes b - b \otimes a, \ \forall a,b \in L$ and $a\otimes [b,c] - [a,b] \otimes c, \ \forall a,b,c \in L$.
Let $\Theta : L \otimes L \to F(L)$ denote the canonical projection.
\end{Def}

\begin{Prop}
The natural map $F(L) \to F(A)$ sending $(a,b) \to \tr(ab)$ is injective.
\end{Prop}
\begin{proof}
Let $\gamma \in F(L)$ be an element in the kernel of $F(L) \to F(A)$. Let us assume that it contains exactly $n+1$ symbols $x$. By using the relations we can assume that it is of the form $\gamma = x \otimes \alpha(x,\ldots,x,y)$ where $\alpha$ is a Lie polynomial in $n$ variables that is linear in the first $n$ variables. Moreover, we can assume that $x_1 \otimes \alpha(x_2,\ldots,x_{n+1},y)$ is symmetric in the $x_i$'s as an element in $F(L)$ and hence also in $F(A)$. Then we obtain
$$
\partial_x \tr(x \alpha(x,\ldots,x,y)) = \alpha(x,\ldots,x,y),
$$
which was assumed to be 0. But then $\gamma = 0$.
\end{proof}




\begin{Rem}
We can visualize monomials in $L$ as a rooted tree where each node has degree 1 or 3, and where one node is marked as the root. To take the brackets of two elements in $L,$ we connect the roots of the two trees to a new root. Then, we can visualize $F(L)$ as unrooted trees, as it sends two rooted trees to an unrooted tree by connecting their roots (or combining the unlabeled node). The relations amount to the facts that after gluing there is no first or second branch ($ a\otimes b- b\otimes a$), and that the newly created edge is indistinguishable from any other edge ($a \otimes [b,c] - [a,b] \otimes c$).
\end{Rem}

Recall that $A$ acts on $L$, $\ad : A \times L \to L ; (a,l) \to \ad_a(l)$ by the adjoint action defined by the recursive formulas
\begin{align*}
\ad_x(l) &= [x,l] &\ad_{xa}(l) &= [x, \ad_a(l)] \\
\ad_y(l) &= [y,l] & \ad_{ya}(l) &= [y, \ad_a(l)].
\end{align*}
Similarly we define the adjoint action on any bimodule, for instance on $A \otimes A$.


Let $\epsilon: A \to \mathbb{R}$ denote the algebra homomorphism that sends $1 \mapsto 1$ and $x,y \mapsto 0$, that is picking out the constant term.
\begin{Def}
For $x_0 \in \{ x,y \}$. We define the \emph{partial} of a Lie term $l \in L$ with respect to $x_0$ by the formula
$$
\partial^L_{x_0}(l) = \partial^1_{x_0}(l) \epsilon( \partial^2_{x_0}(l)),
$$
that is we look for the everything that prefixes $x_0$.
\end{Def}

\begin{Lem}
The $\partial^L$ satisfies the following Leibniz identity,
$$
\partial^L_{x_0}([l,m]) = l \; \partial^L_{x_0}(m) - m \; \partial^L_{x_0}(l)
$$
\end{Lem} 
\begin{proof}
We apply $\op{id} \otimes \epsilon$ to
$$
\partial^A([l,m]) = [l, \partial^A(m)] + [\partial^A(l), m],
$$
and note that $\epsilon(\partial^2(m) l) = \epsilon( \partial^2(l) m) = 0$ (since $m$ and $l$ have no constant terms) to obtain the required identity.
\end{proof}

An alternative definition could be defined according to the following procedure.
First, we sum over all occurences of $x_0$ marking each $x_i = x_0$ in $l$ as $\hat{x_0}$. Using the Jacobi relations, we shuffle the new term with a marked $x_0$ into the form $\ad_{a} (\hat{x_0})$ for some $a_i \in A$ without constant term. Then, we define $\partial_{x_0}^L (l) = a$. To see that this gives the same result, we observe that this procedure satisfies the same Leibniz identity, that is let $l, m \in L$ and let $\hat{l} = \ad_{l_{x_0}} (\hat{x}_0)$ and $\hat{m} = \ad_{m_{x_0}} (\hat{x}_0)$ denote their marked versions. Then marking $[l,m]$ we obtain
\begin{align*}
[l , \hat{m}] + [\hat{l}, m] & = [l, \ad_{m_{x_0}}( \hat{x}_0 )] - [m ,\ad_{l_{x_0}} ( \hat{x}_0 )] \\
& = \ad_{l m_{x_0}} ( \hat{x}_0 ) - \ad_{m l_{x_0}} ( \hat{x}_0 ).
\end{align*}
Replacing $\hat{x}_0$ with the symbol $\otimes$ we obtain the following

\begin{Lem}
For $l \in L$ we have the identity
$$
\partial^A(l) = \ad_{\partial^L(l)} (1 \otimes 1).
$$
\end{Lem}


%

\begin{Def}
For $\gamma \in F(L),$ we define the \emph{partial} of the Lie tree $\gamma$ with respect to $x_0$ by treating $\gamma$ as an object in $F(A)$.
\end{Def}

\begin{Prop}
\label{prop:partialinL}
$\partial_{x_0} (\gamma) \in L$ for all $\gamma \in F(L).$
\end{Prop}

\begin{proof}
We compute $\partial_x \tr(ab)$ where $a,b \in L$. By Lemma \ref{lem:somethingmadeoutofpartial} we have
\begin{align*}
    \partial_x \tr(ab) &= \partial_x^2(ab)\partial_x^1(ab)\\
    &= \partial_x^2(a)b \partial_x^1(a) + \partial_x^2(b)a\partial_x^1(b).
\end{align*}
By the above discussion we have that $\partial^A_x(a) = \ad_{\alpha} (1\otimes 1)$ for $\alpha = \partial^L_x(a)$ we show that $\partial_x^2(a)b \partial_x^2(a)$ lies in $L$ by induction on the length of $\alpha$. That is let $\alpha = z \beta$ for $z \in \{x,y\}$. We compute
\begin{align*}
    \partial_x^2(a)b \partial_x^1(a) &= \beta^2 [-z,b] \beta^1,
\end{align*}
where $\beta^1 \otimes \beta^2 = \ad_{\beta}(1 \otimes 1)$ and thus $\partial_x^2(a)b \partial_x^1(a)$ lies in L by induction. The same argument shows that $\partial_x^2(b)a\partial_x^1(b) \in L$.
\end{proof}





\begin{Def}
A linear map $L \to L$ is called a derivation if
$$
u([a,b]) = [u(a),\;b] + [a,\;u(b)] \quad \forall a,b \in L.
$$
The space of derivations is denoted by $\Der(L)$.
\end{Def}

By Proposition \ref{prop:derivuniqueness} we can uniquely extend any derivation $L \to L$ to a derivation $A \to A$, thus we will identify $\Der(L)$ with a subspace of $\Der(A)$ and we immediately obtain the following


\begin{Prop}
The maps in Proposition \ref{prop:derivuniqueness} restict to
$\Der(L) \cong L \times L.$
\end{Prop}
%
%
%

\begin{Prop}
$\Der_\omega(L) \cong F(L).$
\end{Prop}
\begin{proof}
Since $\Der_\omega(L) \subset \Der_\omega(A)$ and $F(L) \subset F(A),$ we only need to show that our original maps between $\Der_\omega(A)$ and $F(A)$ preserve the respective subsets. By Proposition \ref{prop:partialinL} we have that the map $F(L) \to \Der(A)$ actually lands in $\Der(L)$. The inverse to this map as constructed in the proof of \ref{lem:inverseMap} is given by $u \to \frac{1}{n}\tr(y u(x) - x u(y) )$ and thus lies in the subspace $F(L) \subset F(A)$.



\end{proof}

\section{The elliptic Kashiwara-Vergne Lie algebra}
In this section we recall the definition of the elliptic Kashiwara-Vergne Lie algebra and its basic properties.

\begin{Def}
The \emph{divergence} map $\op{div} : Der(L) \to F(A)$ is defined by the formula
$$
\op{div}(u) = \text{tr}(\partial_x^L(u(x)) + \partial_y^L (u(y))).
$$
\end{Def}

We define the action of $\Der(A)$ on $\tr(A)$ by $u.\tr(a) = \tr(u(a))$.

\begin{Lem}\label{lem:cocycle}
The divergence is a 1-cocycle, that is it satisfies
\begin{align*}
    div([u,v]) = u.div(v) - v.div(u)
\end{align*}
\end{Lem}
\begin{proof}
see \cite{AKKN2} Proposition 3.1 for a proof.
\end{proof}

\begin{Def}The \emph{elliptic Kashiwara-Vergne Lie algebra} is given by$$\mathfrak{krv} = \mathfrak{krv}^{(1,1)} = \{u \in tder(1,1) = \Der(L(x, y)) \ |\ u([x, y]) = 0, \dv(u) = 0\}.$$
\end{Def}

\begin{Rem}
Since we are restricted to only symplectic derivations, we have that graphically, the operation $\dv$ takes a tree to a trace. We observe that the process of taking a tree to a trace can be described graphically as follows: We obtain $u(x)$ by marking each $y$ and reading the resulting rooted trees, while $\partial_x^L$ marks each $x$ on the rooted tree. Then, we have a Lie tree with two markings. We can view this tree as an arrow with Lie terms on the arrow. We expand the Lie terms on the arrow to complete the $\partial_x^L$ operation, and finally connect the two ends of the arrow to form a trace.
\end{Rem}


\begin{Prop}
$\mathfrak{krv}$ is a Lie algebra.
\end{Prop}

\begin{proof}
%
%

Let $u, v \in \mathfrak{krv}$, thus we have that $\dv(u) = \dv(v) = 0$. Then by Lemma \ref{lem:cocycle} we obtain
\begin{align*}
    \dv([u,v]) &= u( \dv(v)) - v( \dv(u)) = 0.
\end{align*}
\end{proof}

\begin{Prop}
$\mathfrak{krv}$ is bigraded, that is $\mathfrak{krv} = \bigoplus_{i\geq 0, j \geq 0} \mathfrak{krv}^{(i,j)}$, where $\mathfrak{krv}^{(i,j)}$ denotes the set of derivations in $\mathfrak{krv}$ whose corresponding trees have $i$ $x$'s and $j$ $y$'s.
\end{Prop}

\begin{proof}
Both operations div and $u([x,y])$ preserve the degree decomposition. 
\end{proof}

\section{Small elements in $\mathfrak{krv}$}
In this section we are studying low-degree elements in $\mathfrak{krv}$. To that purpose let us first recall the graphical interpretation of the Jacobi identity.
\begin{Def}[IHX]
The IHX relation is the graphical representation of the Jacobi identity.

$$[a,[b,c]] \ \ \ \ \ \ +  \ \ \ \ \ \ [c,[a,b]] \ \ \ \ \ \ + \ \ \ \ \ \ [b,[c,a]] \ \ \ \ \ \ = \ \ \ \ 0.$$

\begin{figure}[H]
    \centering
    \includegraphics[width=0.9\textwidth]{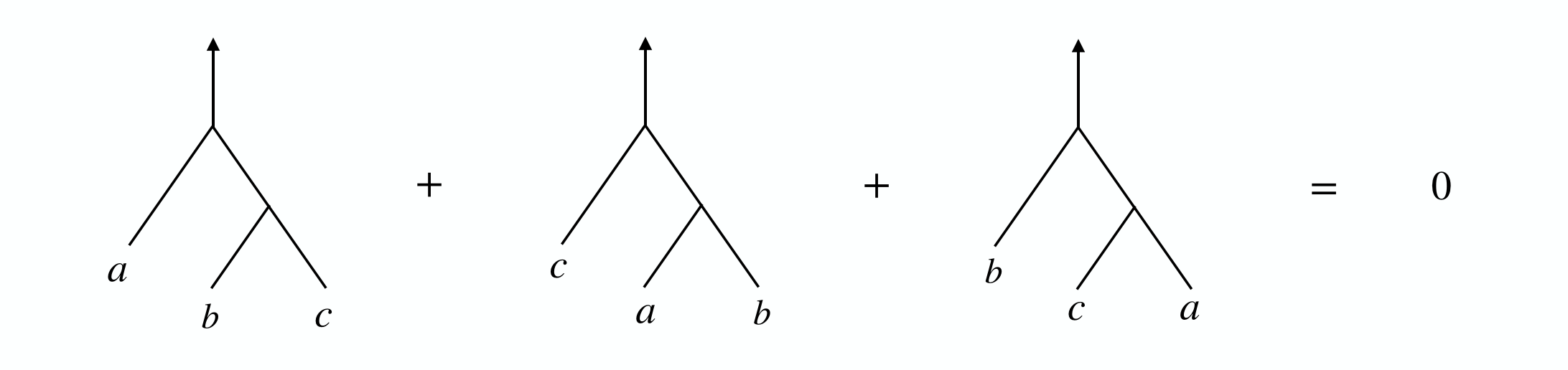}
    \label{fig:ihx1}
\end{figure}

Fixing the positions of $a, b, c$ and the root, we have

\begin{figure}[H]
    \centering
    \includegraphics[width=0.9\textwidth]{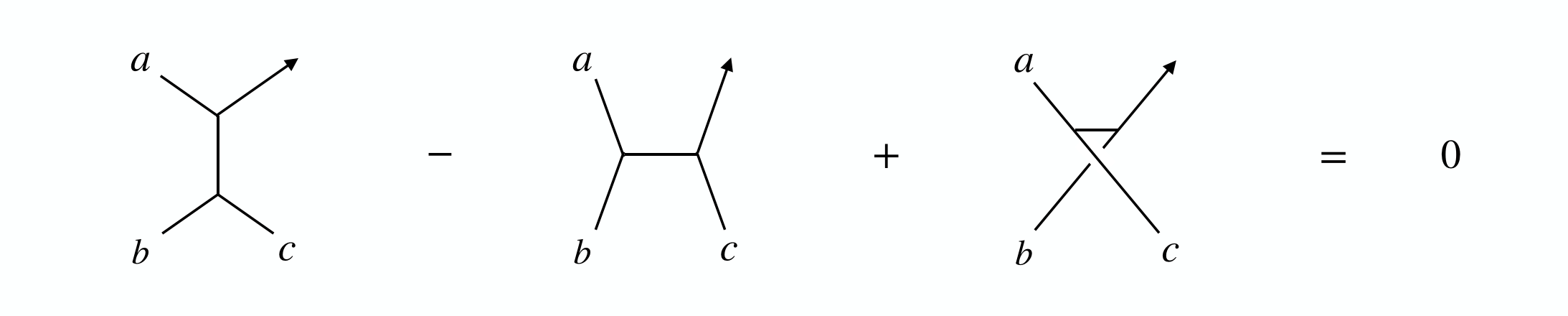}
    \label{fig:ihx2}
\end{figure}
\end{Def}

\begin{Lem}[Birds on a Wire]
Given two points $x_1, x_2$ on a Lie tree, it is always possible to shuffle the tree into the form $\Theta(x_1, \ad_a(x_2))$ for some $a\in A.$ We call this the standard form of representing a tree with two ordered marked roots.
\end{Lem}

\begin{proof}
We always have path connecting $x_1$ and $x_2$. Mark $x_1$ as the head of an arrow, and $x_2$ as the tail of the arrow.  We straighten the path, and position it vertically with $x_1$ on the top. We will have trees growing from the arrow. Using IHX relations, we can reduce any tree by a level unless it only has one layer.

\begin{figure}[H]
    \centering
    \includegraphics[width=0.9\textwidth]{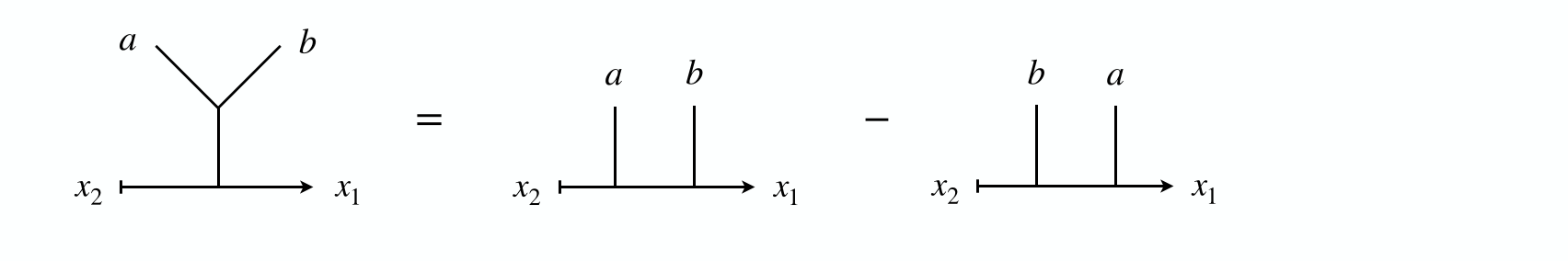}
    \label{fig:ihx3}
\end{figure}

We can also switch a tree from the right side to the left using antisymmetry. Then, we can always reduce the tree to the sum of elements with $x_2$ on the bottom, $x_1$ on top, and individual symbols $a_1, a_2, \ldots a_n$ on the left of the arrow. These elements are then of the desired form  $\Theta(x_1, \ad_a(x_2)).$
\end{proof}

\begin{Lem}
\label{lem:dihed}
For $\Gamma \in F(L)$ with $u_\Gamma \in \Der(L)$ the corresponding derivation, we have
$$
\dv(u_\Gamma) = \tr(g - g^*),
$$
where $g = \partial^L_x \partial^{F(L)}_y (\Gamma)$ and $g \to g^*$ is reversing order and multiplying each symbol by $-1$.
\end{Lem}

\begin{proof}
We have that
\begin{align*}
\dv(u_{\Gamma}) &= \tr(\partial_x^L(u(x)) + \partial_y^L(u(y))) \\
&= \tr(\partial_x^L(\partial_y^{F(L)}(\Gamma)) - \partial_y^L(\partial_x^{F(L)}(\Gamma))).
\end{align*}
Graphically, to obtain $\partial_x^L(\partial_y^{F(L)}(u))$, we first mark a $y$ to option a rooted Lie tree. Then we mark each $x$. We construct an arrow starting at the marked $y$ and ending at the $x.$ We showed by the birds-on-a-wire lemma that we can always reshuffle components on the arrow into the standard form, and read the symbols as terms in $A$. To obtain $\partial_y^L(\partial_x^{F(L)}(u)),$ we perform a very similar process, but switch the order of marking the $x$ and $y$. Then, after we do the same shuffling, we obtain similar results, with the arrow pointing in the different direction. The individual symbols on the arrow now lie on the other side of the arrow and in the opposite order. To move these symbols to the standard form, we use antisymmetry and multiply by $-1$, and they stay in the opposite order. Then, if the original partial we obtain is $a\in A$, our new partial must be $a^*.$
\end{proof}

For small degree it turns out that in the odd case the divergence condition is automatic. More precisely, we have
\begin{Lem}[Small Wheels]
Let $\Gamma \in F(L)$ be of even total degree and of $x$-degree at most 3. Then
$$
\op{div} (\Gamma) = 0.
$$
\end{Lem}

\begin{proof}
The divergence of a derivation has total degree 2 less than that of the derivation. Then, if the derivation is from a tree of even total degree, we have $tr(g^*)$ is $\tr(g)$ in the opposite order with the same sign. Moreover $g$ contains at most 2 $x$'s, and in that case it is symmetric with respect to reversing order.
\end{proof}

%

\subsection{Elements of weight 2}
By the birds-on-a-wire lemma any element of $x$-degree 2 is one of
$$
\delta_{2n} := \Theta(x, \ad_{y^{2n}}(x)),
$$
as any odd degree element is zero by
$$
\Theta(x, \ad_{y^{2n+1}}(x)) = (-1)^{2n+1}\Theta(x, \ad_{y^{2n+1}}(x)) =  (-1)^{2n+1}\Theta(\ad_{y^{2n+1}}(x), x ).
$$
The $\delta_{2n}$ are moreover linearly independent as can be seen by computing
$$
u_{\delta_{2n}}(y) = -\partial_x^{F(L)}(\delta_{2n}) = -(\ad_{y^{2n}}(x) + \ad_{(-y)^{2n}}(x)) = -2 \ad_{y^{2n}}(x).
$$
We have thus shown
\begin{Prop}
The elements $\delta_{2n}$ form a basis for elements of weight 2.
\end{Prop}
%
%

\begin{Prop}
All elements $\delta_{2n}$ are in $\mathfrak{krv}.$
\end{Prop}

\begin{proof}
By the small wheel's lemma, all trees with an even number of roots and only 2 $x$'s have divergence 0.
\end{proof}

\subsection{Elements of weight 3}

By the small wheels lemma we immediately obtain that any even (total degree) element of $x$-degree 3 is contained in $\mathfrak{krv}$. The main result in this section is that the converse is true.
\begin{Thm}
\begin{equation*}
\mathfrak{krv}^{(3,j)} = 
\begin{cases}
	0, &\text{for } j \text{ even} \\
    F(L)^{(3,j)}, &\text{for } j \text{ odd}.
\end{cases}
\end{equation*}
\end{Thm}

We will study the image of $\mathfrak{krv}$ under the map $F(L) \to L$
\begin{align*}
F(L) &\longrightarrow L \\
\Gamma &\longmapsto -u_\Gamma(y) = \partial_x(\Gamma).
\end{align*}
By Lemma \ref{lem:euler} for any element $\Gamma \in F(L)^{(k,j)}$ we have
$$
k \Gamma = \Theta( x, \partial_x( \Gamma)),
$$
and hence the above assigment is injective and (after applying $\partial_x$ to this equation) we see that the image of $F(L)^{(k,\bullet)}$ can identified with the eigenspace to the eigenvalue $k$ of the operator
    \begin{align*}
        \kappa: L &\longrightarrow L \\
        \Gamma &\mapsto \partial_x (\Theta(x \Gamma)).
    \end{align*}

Let us now write an arbitrary element in $L$ of $x$-degree 2 in the form
$$
\sum_{i,j} c_{i,j} [\ad_y^i(x), \ad_y^j(x)],
$$
for some polynomial $P = \sum_{i,j} c_{i,j} X^i Y^j \in \mathbb{R}[X,Y]$ satisfying the antisymmetry condition
\begin{equation}
P(X,Y) = -P(X,Y).
\end{equation}
This sets up a linear bijection between the $x$-degree 2 elements of $L$ and antisymmetric polynomials. We want to study the conditions: (i) $P$ corresponds to an element coming from a tree $\Gamma \in F(L)$ and (ii) the corresponding $\Gamma$ has divergence 0.

\subsection{Condition (i)}
\begin{Lem}
Under the above identification, the operator $\kappa$ corresponds to
$$
P(X,Y) \mapsto P(X,Y) - P(-X-Y,Y)+ P(-X-Y,X)
$$
\end{Lem}
\begin{proof}
It is enough to compute $\kappa([\ad_y^i(x), \ad_y^j(x)]) = \partial_x( \Theta(x, [\ad_y^i(x), \ad_y^j(x)]))$ which corresponds to $P(X,Y) = \tfrac{1}{2}(X^iY^j - X^jY^i)$. Using the marking procedure this gives
\begin{align*}
 \Theta(\hat{x}, [\ad_y^i(x), \ad_y^j(x)]) +  \Theta(x [\ad_y^i(\hat{x}), \ad_y^j(x)]) +  \Theta(x [\ad_y^i(x), \ad_y^j(\hat{x})]) \\
 =\Theta(\hat{x}, [\ad_y^i(x), \ad_y^j(x)]) -  \Theta(\ad_y^i(\hat{x}), [x, \ad_y^j(x)]) +  \Theta(\ad_y^j(\hat{x}), [x, \ad_y^i(x)) \\
 =\Theta(\hat{x}, [\ad_y^i(x), \ad_y^j(x)]) -  \Theta(\hat{x}, \ad_{-y}^i ([x, \ad_y^j(x)])) +  \Theta(\hat{x}), (\ad_{-y}^j([x, \ad_y^i(x)]))
\end{align*}
and thus
$$
\kappa([\ad_y^i(x), \ad_y^j(x)])  = [\ad_y^i(x), \ad_y^j(x)] - \ad_{-y}^i ([x, \ad_y^j(x)]) + \ad_{-y}^j([x, \ad_y^i(x)]),
$$
which corresponds to the polynomial
\begin{eqnarray*}
\tfrac{1}{2}( X^i Y^j - X^j Y^i  - (-X-Y)^iY^j + (-X-Y)^iX^j + (-X-Y)^jY^i - (-X-Y)^jX^i) \\
= P(X,Y) - P(-X-Y,Y)+ P(-X-Y,X),
\end{eqnarray*}
which proves the claim.
\end{proof}

Thus we have shown the following
\begin{Prop}
There is a bijection between elements of $F(L)$ of $x$-degree 3 and the space
\begin{equation}
\label{eq:cond1}
\mathcal{P} = \{ P \in \mathbb{R}[X,Y] \ | \ P(X,Y) = -P(Y,X), \ 2 P(X,Y) = P(-X-Y,X) - P(-X-Y,Y) \}.  
\end{equation}
\end{Prop}

\subsection{Condition (ii): $\dv(u) = 0$}
Next we wish to express the condition $\dv(u) = 0$ under the above identification in terms of the polynomial $P$. To this purpose let us identify the relevant trace space $\tr^{\op{deg}_x = 2}$ with symmetric polynomials under the map
\begin{align*}
\mathbb{R}[X,Y] &\longrightarrow \tr^{\op{deg}_x = 2} \\
X^i Y^j &\longmapsto \tr(xy^ixy^j).
\end{align*}

Let us now compute $\dv(u)$ for $u$ corresponding to $P = -\tfrac{1}{2}X^iYj - X^jY^i$. By Lemma \ref{lem:dihed} we have
$$
\dv(u) = \tr( \partial_y u(y) - (\partial_y u(y))^* ).
$$
By assumption $u(y) = [\ad_y^i(x), \ad_y^j(x)]$. We compute
\begin{align*}
    \partial_y([\ad_y^i(x), \ad_y^j(x)]) = (\ad^i_y x \ad^j_y x - \ad^j_y x \ad^i_y x - \ad^i_y x y^j x + \ad^j_y x y^i x)y^{-1}.
\end{align*}
To take the trace of this expression, we only need to keep track of the number of $y$'s between the $x$'s. Since in the expansion of each term in the expression, the relative positions of the two $x$'s are fixed, we count the number of $y$'s outside of the two $x$'s and between the two $x$'s. We temporarily denote the $y$'s that expand to the outside as $a$, and those on the inside as $b$. This way we obtain the expression
$$
((a-b)^i (b-a)^j - (b-a)^i (a-b)^j - (a-b)^i b^j + b^i (a-b)^j)  a^{-1}.
$$
In terms of $P$ this can be written as
$$
(P(a-b,b) - P(a-b,b-a))/a.
$$
If we further assume that $P$ is even we have $P(x, -x) = P(-x, x)$ while $P(x, -x) = -P(-x, x)$ by antisymmetry. Then, $P(a-b, b-a) = 0,$ so we are left with the term $P(a-b, b)/a$. Under this assumption we that $\dv(u) = 0$ is equivalent to
\begin{equation}
\label{eq:cond2}
\frac{1}{X}P(Y,X-Y) + \frac{1}{Y}P(X,Y-X) = 0,
\end{equation}
where the second term corresponds to $(\partial_y u(y) )^*$.

\subsubsection{Solving the functional equations: $P$ is even}
From equations \eqref{eq:cond1} and \eqref{eq:cond2}, we conclude that even elements in $\mathfrak{krv}^{(3,j)}$ correspond to polynomials $P \in \mathbb{R}[X,Y]$ that satisfy the following conditions:

\label{polyeqns}
\ \ \ -1): $P$ is even

\ \ \ 0): \ $P(X,Y) = - P(Y,X)$

\ \ \ 1): \ $2 P(X,Y) = P(Y, -X-Y) - P(X, -X-Y)$

\ \ \ 2): \ $\frac{1}{X}P(Y,X-Y) + \frac{1}{Y}P(X,Y-X) = 0$

\begin{Prop}
The only polynomial $P$ satisfying the above conditions is $P = 0$.
\end{Prop}

\begin{proof}
We show this by assuming that a $P$ satisfies the equations, and repeatedly factoring $P,$ updating the conditions, until the conditions repeat. We conclude that $P$ must have infinitely many factors so it must be 0.

First, we notice that $P(x,y)$ is divisible by $x,y,x+y$ and $x-y$. Then, through equation 2) we also have that $P(x,y)$ is divisible by $2x+y$ and $x+2y.$ Then, let $P(x,y) = (x)(y)(x+y)(x-y)(2x+y)(x+2y)P_2(x,y).$

We apply this substitution, and update the four conditions to:

\ \ \ -1): $P_2$ is even

\ \ \ 0): \ $P_2(X,Y) = P_2(Y,X)$

\ \ \ 1): \ $2 P_2(X,Y) = P_2(Y, -X-Y) + P_2(X, -X-Y)$

\ \ \ 2): \ $\frac{1}{X}P_2(Y,X-Y) - \frac{1}{Y}P_2(X,Y-X) = 0$

Then, we notice that $P_2(x,y)$ is divisible by $x, y, x+y.$ Let $P_2(x,y) = xy(x+y)P_3(x,y).$

We apply the new substitution and update the conditions again to:

\ \ \ -1): $P_3$ is odd

\ \ \ 0): \ $P_3(X,Y) = P_3(Y,X)$

\ \ \ 1): \ $2 P_3(X,Y) = P_3(Y, -X-Y) + P_3(X, -X-Y)$

\ \ \ 2): \ $\frac{1}{X}P_3(Y,X-Y) + \frac{1}{Y}P_3(X,Y-X) = 0$

We again get $x, y, x+y$ divide $P_3(x,y).$ Let $P_3(x,y) = xy(x+y)P_4(x,y).$

We apply the new substitution and update the conditions again to
\ \ \ -1): $P_4$ is even

\ \ \ 0): \ $P_4(X,Y) = P_4(Y,X)$

\ \ \ 1): \ $2 P_4(X,Y) = P_4(Y, -X-Y) + P_4(X, -X-Y)$

\ \ \ 2): \ $\frac{1}{X}P_4(Y,X-Y) - \frac{1}{Y}P_4(X,Y-X) = 0$

We notice that the conditions are the same as those for $P_2.$ Then, the factoring process repeats, and the original $P$ must have infinitely many factors, so it must be 0.
\end{proof}

\subsubsection{P is odd}

When $P$ is odd, we do not have condition (ii). Then, the dimension of $\mathfrak{krv}^{3,j}$ for $j$ odd is the dimension of $F(L)^{(3,j)}$, that is the number of Jacobi trees with 3 $x$'s and $j$ $y$'s.

\begin{Prop}
The dimension of Jacobi trees with 3 $x$'s and $m$ $y$'s is given by $\lfloor\frac{m-1}{2}\rfloor - \lfloor\frac{m-1}{3}\rfloor.$
\end{Prop}

\begin{proof}
Any trees with 3 $x$’s is a sum of elements of the form $\Theta(ad_{y^i}(x), [ad_{y^j}(x), ad_{y^k}(x)])$.
We can encode these elements as a polynomial in $P = \mathbb{R}[x,y,z].$ We have a map $\alpha: P \to F(L)$, given by the formula
$$
x^i y^j z^k \mapsto \Theta(\ad_y^i x [\ad_y^j x, \ad_y^k x]).
$$
We observe the antisymmetry in trees translate to the polynomial being totally anti-symmetric. We then factor our polynomial into $(x-y)(y-z)(z-x)p(x,y,z)$, where $p$ is a symmetric polynomial.
The IHX relations in trees translate to that $x^{i+1} y^j z^k + x^i y^{j+1} z^k + x^i y^j z^{k+1} = x^i y^j z^k (x+y+z)$ gets sent to zero.
Therefore, anything of the form $(x+y+z)f(x,y,x)$ gets sent to zero. A symmetric polynomial $p$ in three variables can be expressed in terms of the elementary symmetric polynomials $x+y+z, \ xz+xy+yz, \ xyz$. Since all multiples of $(x+y+z)$ are sent so $0$, we can assume that $p$ is a polynomial in $xz+xy+yz$ and $xyz$. The degree of the polynomial $(x-y)(y-z)(z-x)p(x,y,z)$ gives the  $y$-degree of the tree. Then, the dimension of the set of polynomials corresponding to trees of $m$ $y$’s is given by the number of solutions to $\deg((xz+ xy+ yz)^a (xyz)^b) + 3 = m,$ or $2a + 3b = m-3$ for integers $a, b \geq 0$. This equation has $\lfloor\frac{m-1}{2}\rfloor - \lfloor\frac{m-1}{3}\rfloor$ solutions.
\end{proof}


\end{document}